\documentclass[12pt]{amsart}
\usepackage{amsfonts}
\usepackage{amsfonts,latexsym,rawfonts,amsmath,amssymb,amsthm}
\usepackage[plainpages=false]{hyperref}
\usepackage{graphicx}

\numberwithin{equation}{section}

\RequirePackage{color}
\theoremstyle{plain}

\newtheorem{theorem}{Theorem}[section]

\newtheorem{lemma}[theorem]{Lemma}
\newtheorem{proposition}[theorem]{Proposition}
\newtheorem{remark}[theorem]{Remark}

\newcommand{\beq}{\begin{equation}}
\newcommand{\eeq}{\end{equation}}
\newcommand{\beqs}{\begin{eqnarray*}}
\newcommand{\eeqs}{\end{eqnarray*}}
\newcommand{\beqn}{\begin{eqnarray}}
\newcommand{\eeqn}{\end{eqnarray}}
\newcommand{\beqa}{\begin{array}}
\newcommand{\eeqa}{\end{array}}

\def\phi{\varphi}

\begin{document}
\title[Fully nonlinear equations of Krylov type]{Fully nonlinear equations of Krylov type on Riemannian manifolds
with negative curvature}

\author{Li Chen}
\address{Faculty of Mathematics and Statistics, Hubei Key Laboratory of Applied Mathematics, Hubei University,  Wuhan 430062, P.R. China}
\email{chenli@hubu.edu.cn}

\author{Yan He$^\ast$}
\address{Faculty of Mathematics and Statistics, Hubei Key Laboratory of Applied Mathematics, Hubei University,  Wuhan 430062, P.R. China}
\email{helenaig@hotmail.com}



\thanks{$\ast$ Corresponding author}

\begin{abstract}
In this paper, we consider fully nonlinear equations of Krylov type
on Riemannian manifolds with negative curvature which naturally
arise in conformal geometry. Moreover, we prove the a priori
estimates for solutions to these equations and establish the
existence results. Our results can be viewed as an extension of
previous results given by Gursky-Viaclovsky and Li-Sheng.
\end{abstract}

\maketitle

\section{Introduction}

Let $(M, g_0)$ be a smooth closed Riemannian manifold of dimension
$n\ge3$ and  $[g_0]$ denote the conformal class of $g_0$ on $M$, the
well-known $\sigma_k$-Yamabe problem is of finding a metric $g \in
[g_0]$ satisfies the following equation on $M$
\begin{equation}
\sigma_k(A_g)=constant,\label{eq000}
\end{equation}
where
\begin{eqnarray*}
A_{g}={\frac{1}{{n-2}}} \left( {Ric}_{g}-
{\frac{{{R}_{g}}}{{2(n-1)}}}g \right)
\end{eqnarray*}
is the Schouten tensor of $g$, ${Ric}_{g}$ and ${R}_{g}$ are the
Ricci and scalar curvatures of $g$ respectively, we denote by
$\sigma_k(\lambda)$ the $k$-th elementary symmetric polynomial
$$\sigma_k(\lambda)=\sum_{1\le i_1<\cdots<i_k\le n}\lambda_{i_1}\cdots\lambda_{i_k},
 \quad \lambda=(\lambda_1, \lambda_2, ..., \lambda_n)\in \mathbb{R}^n,$$
for $1\leq k\leq n$ and we set $\sigma_0(\lambda)=1$, and
$\sigma_k({A_{g}})$ means that $k$-th elementary symmetric
polynomial $\sigma_k$ is applied to the eigenvalues of $g^{-1}\cdot
{A_g}$.

For $k=1$, the equation \eqref{eq000} is just the classical Yamabe
equation which has been solved by Yamabe \cite{Y60}, Trudinger
\cite{Tr68} , Aubin \cite{Au76} and Schoen \cite{S84}. The fully
nonlinear elliptic equation \eqref{eq000} ($k\geq 2$) has been
studied extensively after the pioneering works of Viaclovsky
\cite{V00, V00-1, V02}. Under the assumption that the eigenvalues
$\lambda(A_{g_0})$ of the matrix $g^{-1}\cdot {A_{g_0}}$ belong to
$\Gamma_k$ with
$$\Gamma _{k}=\{\lambda =(\lambda _{1},\dots ,\lambda _{n})\in
\mathbb{R}^{n}|\quad \sigma _{j}\left( \lambda \right) >0, \ \forall
1\leq j\leq k\},$$ the $\sigma_k$-Yamabe equation \eqref{eq000} has
been solved for either $k=2$, or $k\geq \frac{n}{2}$, or $M$ being
locally conformally flat by the works of Chang-Gursky-Yang
\cite{CGY02a, CGY02b}, Guan-Wang \cite{GW03b}, Li-Li \cite{LL03},
Gursky and Viaclovsky \cite{GV07}, Li-Nguyen\cite{LN14}, Ge-Wang
\cite{GW06}, Sheng-Trudinger-Wang \cite{STW07} and
Brendle-Viaclovsky \cite{BV04}. For related results, see \cite{BM09,
Chen05, Chen07, Chen09, E92, GW03a, GW03b, GW13, GLW10, GW07, HSh11,
HSh13, JLL07, JT17,JT18,JT19, KMS09, LL06, LL03, STW07, W06, ShY13}
and so on.

Now, we turn into the negative curvature case. Gursky-Viaclovsky
\cite{GV03a} introduced the modified Schouten tensor for $\tau \in
\mathbb{R}$
\begin{eqnarray*}
A^{\tau}_{g}={\frac{1}{{n-2}}} \left( {Ric}_{g}-
{\frac{{\tau{R}_{g}}}{{2(n-1)}}}g \right).
\end{eqnarray*}
When $\tau=1$, $A^{1}_{g}$ is just the Schouten tensor $A_g$. Let
$\lambda(-A^{\tau}_{g_0})$ be eigenvalues of the matrix
$g_{0}^{-1}\cdot (-{A^{\tau}_{g_0}})$ and
$\sigma_k(-A^{\tau}_{g})=\sigma_k(\lambda(-A^{\tau}_{g}))$.
Gursky-Viaclovsky \cite{GV03a} proved that for $\tau<1$ and
$\lambda(-A^{\tau}_{g_0}) \in \Gamma_k$, there exists a unique
conformal metric $g\in [g_0]$ satisfying
\begin{equation}\label{A-0}
\sigma_k(-{A^{\tau}_{g}})=f(x)
\end{equation}
for any smooth positive function $f(x)$ on $M$. A parabolic proof
was later given by Li-Sheng \cite{LiSh05}. Since the equation
\eqref{A-0} is not necessarily elliptic for $\tau>1$ and the $C^2$
estimate does not work for the equation \eqref{A-0} with $\tau=1$ as
noted previously in \cite{V00-1}, the restriction $\tau<1$ must be
made in \cite{GV03a} and \cite{LiSh05}. Those works motivated the
later research on the equation \eqref{A-0} with $\tau$ the boundary
conditions \cite{GSW11,G07,G08}. See \cite{ShZ07} for related
research and \cite{Sui17} for the recent progress on noncompact
manifolds.

In this paper, we study an extension of the equation \eqref{A-0}.
Let $(M, g_0)$ be a smooth closed Riemannian manifold of dimension
$n\ge3$ and $[g_0]$ denote the conformal class of $g_0$ on $M$, we
want to find a metric $g \in [g_0]$ satisfies the following equation
on $M$
\begin{equation}\label{proric}
\sigma_k(-{A^{\tau}_{g}})+\alpha(x)\sigma_{k-1}(-{A^{\tau}_{g}})=\sum_{l=0}^{k-2}\alpha_l(x)\sigma_l(-{A^{\tau}_{g}}),
\quad 3\leq k\leq n.
\end{equation}

The following is our main theorem.

\begin{theorem}\label{main2}
Assume that $\tau<1$ and  $\lambda(-A^{\tau}_{g_0}) \in \Gamma_k$.
Let $\alpha_l(x)$ with $0\leq l\leq k-2$ and $\alpha(x)$ be smooth
functions on $M$. Then there exists a conformal metric $g \in [g_0]$
satisfies equation (\ref{proric}) if  $\alpha_l(x)>0$ for all $0\leq
l\leq k-2$ and $x \in M$.
\end{theorem}

\begin{remark}
We proved the existence of solutions to the equation \eqref{proric}
in Theorem \ref{main2} without the sign requirement for the
coefficient function $\alpha(x)$ in the equation \eqref{proric}.
This type of the equation was first considered by Guan-Zhang
\cite{GZ19}.
\end{remark}

\begin{remark}
In fact, Theorem \ref{main2} holds for $\alpha_l(x)$ with $0\leq
l\leq k-2$ satisfying either $\alpha_l(x)>0$ for all $x \in M$, or
$\alpha_l\equiv 0$, but at least one of $\alpha_l(x)$ with $0\leq
l\leq k-2$ is positive on $M$. Thus,  Theorem \ref{main2} recovers
the previous results proved by Gursky-Viaclovsky \cite{GV03a} and
Li-Sheng \cite{LiSh05}.
\end{remark}

When $\alpha_l\equiv 0$ for $1\leq l\leq k-2$ and $\tau=0$, the
equation \eqref{proric} was considered by the authors with Guo
\cite{CGH}. In fact, \eqref{proric} is the equation of\textbf{
Krylov type} which has been introduced and studied by Krylov in
\cite{Kry95} thirty years ago, and can be seen the extension of the
landmark work \cite{CNS84,CNS85} on the Hessian equation
investigated by Caffarelli-Nirenberg-Spruck. In detail, Krylov
studied the general Hessian equation
\begin{eqnarray}\label{Kr}
\sigma_k(D^2u)+\alpha(x)\sigma_{k-1}(D^2u)= \sum_{l=0}^{k-2}
\alpha_l(x)\sigma_l(D^2u), \quad x \in \Omega\subset \mathbb{R}^n.
\end{eqnarray}
In particular, he observed that if $\alpha(x) \leq 0$ and
$\alpha_l(x)\geq 0$ for $0\leq l\leq k-2$, the natural admissible
cone to make the equation \eqref{Kr} elliptic is also the
$\Gamma_k$-cone which is the same as the $k$-Hessian equation case.
Recently, Guan-Zhang \cite{GZ19} studied the equation of Krylov type
in the problem of prescribing convex combination of area measures
\cite{Sch13}
\begin{eqnarray}\label{CM}
\sigma_k(D^2u+uI)+\alpha(x)\sigma_{k-1}(D^2u+uI)=\sum_{l=0}^{k-2}\alpha_l(x)\sigma_l(D^2u+uI)\quad
\mbox{on} \quad \mathbb{S}^n,
\end{eqnarray}
with $\alpha_{l}(x)\geq 0$ for $0\leq l\leq k-2$, but without the
sign requirement for the coefficient function $\alpha(x)$. In this
case, they observed that the proper admissible set of solutions of
the equation \eqref{CM} is $\Gamma_{k-1}$, not $\Gamma_k$. Based on
this important observation, they also studied Krylov's equation
\eqref{Kr} in $\Gamma_{k-1}$. In fact, such type of the equations
with its structure as a combination of elementary symmetric
functions in fact arise naturally from many important geometric
problems, such as the so-called Fu-Yau equation arising from the
study of the Hull-Strominger system in theoretical physics (see
Fu-Yau \cite{Fu07, Fu08} and Phong-Picard-Zhang \cite{Ph17, Ph19,
Ph20}), the special Lagrangian equations introduced by Harvey-Lawson
\cite{HL}, and so on.

The present paper is built up as follows. In Sect. 2 we start with
some background. In Sect. 3, we obtain the a priori estimates.  We
will prove Theorem \ref{main2} in Sect. 4.

\vskip30pt

\section{Preliminaries}

Let $(M, g_0)$ be a smooth closed Riemannian manifold of dimension
$n\ge3$ with Levi-Civita connection $\nabla$. For later convenience,
we first state our conventions on derivative notation. For a $(0,
r)$-tensor field $V$ on $M$, its covariant derivative $\nabla V$ is
a $(0, r+1)$-tensor field whose coordinate expression is denoted by
$$\nabla V=(V_{k_{1}\cdot\cdot\cdot
k_{r} i}).$$ Similarly, the coordinate expression of the second
covariant derivative of $V$  is denoted by
$$\nabla^2 V=(V_{k_{1}\cdot\cdot\cdot
k_{r} ij}),$$ and so on for the higher order covariant derivatives.
Under the conformal transformation ${g}=\exp{(2u)}g_0$, the Ricci
curvature of $g$ is given by the formula (see \cite{GV03a})
\begin{eqnarray*}
-A^{\tau}_{g}=\nabla^2 u +\frac{1-\tau}{n-2} \Delta u g_0 +
\frac{2-\tau}{2}|\nabla u|^2 g_0- du\otimes du -A^{\tau}_{g_0},
\end{eqnarray*}
where (and throughout the paper) $\Delta u$ and $\nabla^2 u$ denote
the Laplacian and Hessian of $u$ with respect to the background
metric $g_0$. Consequently, the proof of Theorem \ref{main2} reduces
to finding a solution $u\in C^{\infty}(M)$ to the partial
differential equation of second order
\begin{eqnarray}\label{u-Eq}
\frac{\sigma_k(U)}{\sigma_{k-1}(U)}-\sum_{l=0}^{k-2}\alpha_l(x)\exp({2(k-l)u})\frac{\sigma_{l}(U)}
{\sigma_{k-1}(U)}=-\alpha(x) \exp({2u}),\label{eq99}
\end{eqnarray}
where
\begin{eqnarray*}
U=\nabla^2 u +\frac{1-\tau}{n-2} \triangle u g_0 +
\frac{2-\tau}{2}|\nabla u|^2 g_0- du\otimes du -A^{\tau}_{g_0},
\end{eqnarray*}
and $\sigma_k({U})$ means that $k$-th elementary symmetric
polynomial $\sigma_k$ is applied to the eigenvalues of
$g_{0}^{-1}\cdot U$. To solve the equation \eqref{u-Eq}, we consider
a family of equations
\begin{eqnarray}\label{t-u-Eq}
G(U^t)&:=&\frac{\sigma_k\big(U^t\big)}
{\sigma_{k-1}\big(U^t\big)}\notag\allowdisplaybreaks-\frac{\sum_{l=0}^{k-2}\Big([(1-t)c+t\alpha_l(x)]\exp[{2(k-l)u}]\sigma_{l}\big(U^t\big)\Big)}
{\sigma_{k-1}\big(U^t\big)}\notag\allowdisplaybreaks\\
&=&-t\alpha(x) \exp({2u}),\label{eq99}
\end{eqnarray}
where  $ t\in [0, 1]$,
$c=\frac{\sigma_{k}(e)}{\sum_{l=0}^{k-2}\sigma_l(e)}$, $e=(1,\cdots,
1)$, and
\[U^t=\nabla^2 u +\frac{1-\tau}{n-2} \triangle u g_0 +
\frac{2-\tau}{2}|\nabla u|^2 g_0- du\otimes du
-tA^{\tau}_{g_0}+(1-t)g_0.\] Here $\sigma_k\big(U^t\big)$ means that
$k$-th elementary symmetric polynomial $\sigma_k$ is applied to the
eigenvalues of $g_{0}^{-1}\cdot U^t$.

Now we denote by $\lambda(U^t)$ the eigenvalues of the matrix
$g^{-1}_{0}\cdot U^t$ throughout the paper. The following
proposition says the proper admissible set for the solutions of
\eqref{t-u-Eq} is $\Gamma_{k-1}$ which was first observed by
Guan-Zhang \cite{GZ19}.

\begin{proposition}\label{ellipticconcave}
Assume $\tau<1$, then the operator
\begin{eqnarray*}
G(U^t)&:=&\frac{\sigma_k\big(U^t\big)}
{\sigma_{k-1}\big(U^t\big)}\notag\allowdisplaybreaks-\frac{\sum_{l=0}^{k-2}\Big([(1-t)c+t\alpha_l(x)]
\exp[{2(k-l)u}]\sigma_{l}\big(U^t\big)\Big)}
{\sigma_{k-1}\big(U^t\big)}\notag\allowdisplaybreaks
\end{eqnarray*}
is elliptic and concave about $u$ if $\lambda(U^t) \in
\Gamma_{k-1}$, and $\alpha_l(x)\in C^\infty(M)$ is nonnegative for
$0\leq l\leq k-2$.
\end{proposition}
\begin{proof}
The proof is almost the same to that of Proposition 2.2 in
\cite{GZ19}, so we omit it.
\end{proof}

\vskip30pt

\section{The a priori estimates}

\subsection{$C^0$ estimate}

We begin with an important property of $\sigma_k$.

\begin{lemma}\label{A+B}
Let $A$ and $B$ be symmetric $n\times n$ matrices and $0\leq l<k\leq
n$.

(1) Assume that $A$ is positive semi-definite, $B\in \Gamma_{k-1}$ ,
and $A+B \in \Gamma_{k-1}$. Then, we have
\begin{equation*}
\frac{\sigma_k}{\sigma_{k-1}}(A+B)
\geq\frac{\sigma_k}{\sigma_{k-1}}(B)
\end{equation*}
and
\begin{equation*}
\Big(\frac{\sigma_{k-1}}{\sigma_{l}}\Big)^{\frac{1}{k-1-l}}(A+B)
\geq\Big(\frac{\sigma_{k-1}}{\sigma_{l}}\Big)^{\frac{1}{k-1-l}}(B).
\end{equation*}

(2) Assume that $A$ is negative semi-definite, $B\in \Gamma_{k-1}$ ,
and $A+B \in \Gamma_{k-1}$. Then, we have
\begin{equation*}
\frac{\sigma_k}{\sigma_{k-1}}(A+B)
\leq\frac{\sigma_k}{\sigma_{k-1}}(B)
\end{equation*}
and
\begin{equation*}
\Big(\frac{\sigma_{k-1}}{\sigma_{l}}\Big)^{\frac{1}{k-1-l}}(A+B)
\leq\Big(\frac{\sigma_{k-1}}{\sigma_{l}}\Big)^{\frac{1}{k-1-l}}(B).
\end{equation*}
\end{lemma}

\begin{proof}
Since (2) can be easily proved by applying (1) for matrices $-A$ and
$A+B$, it is sufficient to prove (1). We know from the concavity of
$\frac{\sigma_k}{\sigma_{k-1}}$ in the convex cone $\Gamma_{k-1}$
(see \cite{HS99})
\begin{equation*}
\frac{\sigma_k}{\sigma_{k-1}}(A+B)
\geq\frac{\sigma_k}{\sigma_{k-1}}(A)+\frac{\sigma_k}{\sigma_{k-1}}(B)
\end{equation*}
for $A\in \Gamma_{k-1}$ and $B \in \Gamma_{k-1}$, which implies in
view of the positive semi-definite of $A$
\begin{equation*}
\frac{\sigma_k}{\sigma_{k-1}}(A+B)
\geq\frac{\sigma_k}{\sigma_{k-1}}(B).
\end{equation*}
So, we complete the proof of the first inequality in (1). The second
inequality in (1) can be proved similarly if we notice that
$\Big[\frac{\sigma_{k-1}}{\sigma_{l}}\Big]^{\frac{1}{k-1-l}}$ is
concave in $\Gamma_{k-1}$ for $0\leq l<k-1$ (see Chapter XV in
\cite{Li96}).
\end{proof}

With the help of Lemma \ref{A+B}, $C^0$ estimate can be obtained
consequently.

\begin{lemma}\label{estn0}
Assume that $\tau<1$ and $\lambda(-A^{\tau}_{g_0})\in\Gamma_{k}$.
Let $\alpha_l(x)$ be a positive smooth function on $M$ for all
$0\leq l\leq k-2$ and $\alpha(x)$ be a smooth function on $M$.
Suppose $u$ is a smooth solution of (\ref{eq99}) with $\lambda(U^t)
$$\in\Gamma_{k-1}$. Then there exists a constant $C$ depending on $\tau$,
$g_0$, $||\alpha||_{C^0(M)}$, $||\alpha_l||_{C^0(M)}$ and
$\inf_{M}\alpha_l$ with $0\leq l\leq k-2$, such that
\begin{equation}
\sup_M| u|\le C.\label{C0bdn}
\end{equation}
\end{lemma}

\begin{proof}
Suppose the maximum point of $u$ is attained at $x_1$. Thus
$\nabla^2 u(x_1)$ is negative semi-definite and $\nabla u(x_1)=0$
which implies
\begin{equation*}
\nabla^2 u +\frac{1-\tau}{n-2} \triangle u g_0 +
\frac{2-\tau}{2}|\nabla u|^2 g_0- du\otimes du
\end{equation*}
is negative semi-definite at $x_1$ if $\tau<1$. Thus, we arrive at
$x_1$ from Lemma \ref{A+B}
\begin{equation}\label{C0-1}
\frac{\sigma_k(B_{g_0})}{\sigma_{k-1}(B_{g_0})}
\ge\frac{\sigma_k(U^t)}{\sigma_{k-1}(U^t)},
\end{equation}
and
\begin{equation}\label{C0-2}
\frac{\sigma_{l}(B_{g_0})}{\sigma_{k-1}(B_{g_0})}
\le\frac{\sigma_{l}(U^t)}{\sigma_{k-1}(U^t)}
\end{equation}
for $0\leq l<k-1$, where $B_{g_0}=-tA^{\tau}_{g_0}+(1-t)g_0$.
Plugging the inequalities \eqref{C0-1} and \eqref{C0-2} into the
equation \eqref{t-u-Eq} arrives at $x_1$
\begin{equation*}
 C+C \exp{(2u(x_1))} \ge C \sum_{l=0}^{k-2}\exp{[2(k-l)u(x_1)]}.
\end{equation*}
Thus,
$$\sup_{x \in M}u(x)\leq C.$$
Similarly, we have at one minimum point $x_2$ of $u$
\begin{equation}\label{C0-3}
\frac{\sigma_k(B_{g_0})}{\sigma_{k-1}(B_{g_0})}-C \exp{(2u(x_2))}
\leq C \sum_{l=0}^{k-2}\exp{[2(k-l)u(x_2)]}.
\end{equation}
Since $\lambda(-A^{\tau}_{g_0})\in\Gamma_{k}$, we can obtain
$\frac{\sigma_k(B_{g_0})}{\sigma_{k-1}(B_{g_0})}>0$. Thus, we can
conclude from \eqref{C0-3}
$$\inf_{x \in M}u(x)\geq -C.$$
So, the proof is complete.
\end{proof}

\begin{remark}
In fact, to get an upper bound of $u$, we only need
$\lambda(-A^{\tau}_{g_0})\in\Gamma_{k-1}$. However,
$\lambda(-A^{\tau}_{g_0})\in\Gamma_{k}$ is necessary to show that
$u$ is bounded from below.
\end{remark}

\subsection{Gradient estimate}

For the convenience, we will denote by
\begin{equation*}
G_k(U^t)=\frac{\sigma_k(U^t)}{\sigma_{k-1}(U^t)}, \quad
G_l(U^t)=-\frac{\sigma_{l}(U^t)}{\sigma_{k-1}(U^t)}\quad \mbox{for}
\quad 0\leq l\leq k-2,
\end{equation*}
and
\begin{eqnarray*}
\beta_{l}(x, u, t)=[(1-t)c+t\alpha_l(x)]\exp[{2(k-l)u}]\quad
\mbox{for} \quad 0\leq l\leq k-2.
\end{eqnarray*}
We further denote by
\begin{eqnarray*}
G^{ij}=\frac{\partial G}{\partial U^{t}_{ij}}, \quad
G_k^{ij}=\frac{\partial G_k}{\partial U^{t}_{ij}}, \quad
G_l^{ij}=\frac{\partial G_l}{\partial U^{t}_{ij}} \quad \mbox{for}
\quad 0\leq l\leq k-2,
\end{eqnarray*}
and
\begin{eqnarray*}
G_k^{ij, rs}=\frac{\partial^2 G_k}{\partial U^{t}_{ij}\partial
U^{t}_{rs}}, \quad G_l^{ij, rs}=\frac{\partial^2 G_l}{\partial
U^{t}_{ij}\partial U^{t}_{rs}} \quad \mbox{for} \quad 0\leq l\leq
k-2.
\end{eqnarray*}

\begin{lemma}\label{lem2.6}
Assume $\alpha_l(x)>0$ for all $0 \leq l \leq k-2$ and $x \in M$,
and $u$ is a smooth solution of (\ref{eq99}) with $\lambda(U^t) \in
\Gamma_{k-1}$, then we have
\begin{eqnarray}\label{G_l-B}
\label{2.9}& 0 <\frac{\sigma_{l}(U^t)}{\sigma_{k-1}(U^t)} \leq C, ~~
0 \leq l \leq k-2,
\end{eqnarray}
where the constant $C$ depends only on $n, k,$ $\sup_{M} u$ and
$\inf_{M}\alpha_l$ with $0\leq l\leq k-2$.
\end{lemma}
\begin{proof}
Firstly, if $\frac{\sigma_k}{\sigma_{k-1}}\leq 1$, then we get from
the equation \eqref{eq99}
\begin{equation*}
\beta_l \frac{\sigma_l}{\sigma_{k-1}} \leq
\frac{\sigma_k}{\sigma_{k-1}}+t\alpha(x)\exp({2u})\leq 1+C, ~~0 \leq
l \leq k-2.
\end{equation*}
Thus,
\begin{equation*}
\frac{\sigma_l}{\sigma_{k-1}}\leq \frac{1+C}{\inf_{M}\beta_l }, ~~0
\leq l \leq k-2.
\end{equation*}
Secondly, if $\frac{\sigma_k}{\sigma_{k-1}} > 1$, i.e.
$\frac{\sigma_{k-1}}{\sigma_{k}} < 1$. We can get for $0 \leq l \leq
k-2$ by the Newton-MacLaurin inequality \cite{Tr90, LT94}
\begin{equation*}
\frac{\sigma_l}{\sigma_{k-1}}\leq
\frac{(C_n^k)^{k-1-l}C_n^l}{(C_n^{k-1})^{k-l}}(\frac{\sigma_{k-1}}{\sigma_k})^{k-1-l}
\leq \frac{(C_n^k)^{k-1-l}C_n^l}{(C_n^{k-1})^{k-l}} \leq C(n,k).
\end{equation*}
So, the result follows.
\end{proof}

\begin{lemma}\label{C2-le}
Assume $u$ is a smooth solution of (\ref{eq99}) with $\lambda(U^t)
\in \Gamma_{k-1}$ and $\alpha_l(x)>0$ with $0 \leq l \leq k-2$, then
we have
\begin{eqnarray}
G^{ij}U^{t}_{ij}\geq-t\alpha(x) \exp{(2u)},\label{100}
\end{eqnarray}
\begin{eqnarray}
G^{ij}(g_0)_{ij}\geq\frac{n-k+1}{k},\label{100-1}
\end{eqnarray}
\begin{eqnarray}
G^{ij}U^{t}_{ijp}+\sum_{l=0}^{k-2}\big[\beta_l(x, u, t)\big]_pG_l
=-[t\alpha(x) \exp{(2u)}]_p,\label{99}
\end{eqnarray}
and
\begin{eqnarray}
G^{ij}U^t_{ijpp}\notag&\geq&-\sum_{l=0}^{k-2}\frac{1}{1+\frac{1}{k-1-l}}
\frac{([\beta_l(x, u,
t)]_p)^2}{\beta_l}G_l\\&&-\sum_{l=0}^{k-2}\big[\beta_{l}(x, u,
t)\big]_{pp}G_l-[t\alpha(x) \exp{(2u)}]_{pp}.\label{94}
\end{eqnarray}
\end{lemma}

\begin{proof}
(1) By direct calculation, we have
\begin{eqnarray*}
&&G^{ij}U^{t}_{ij}=G_k^{ij}U^{t}_{ij}+\sum_{l=0}^{k-2}\beta_l(x, u, t)\sum_{i,j}G_l^{ij}U^{t}_{ij}\notag\\
&=&G_k+\sum_{l=0}^{k-2}(l-k+1)\beta_l(x, u, t)G_l\notag\\
&\ge&G=-t\alpha(x) \exp{(2u)}.
\end{eqnarray*}
The first inequality follows consequently.

(2) See page 12 in \cite{GZ19} for the proof of the second
inequality. So we just outline the proof by the following simple
calculation
\begin{eqnarray*}
G^{ij}(g_0)_{ij}&=&G_{k}^{ij}(g_0)_{ij}+\sum_{l=0}^{k-1}\beta_l
G_{l}^{ij}(g_0)_{ij}\\&\geq&
G_{k}^{ij}(g_0)_{ij}\\&\geq&\frac{n-k+1}{k},
\end{eqnarray*}
where we get the last inequality from the following inequality
\begin{eqnarray*}
\sum_{i=1}^{n}\frac{\partial(\frac{\sigma_{k}}{\sigma_{k-1}})}
{\partial \lambda_i}\geq \frac{n-k+1}{k}
\end{eqnarray*}
for $\lambda \in \Gamma_{k-1}$ (see Lemma 2.2.19 in \cite{Ger06}).

(3) Differentiating the equation (\ref{t-u-Eq}) arrives
\begin{eqnarray*}
G^{ij}U^{t}_{ijp}+\sum_{l=0}^{k-2}\big[\beta_l(x, u, t)\big]_pG_l
=-[t\alpha \exp{(2u)}]_p.
\end{eqnarray*}
So the third equality follows.

(4) Differentiating the equation (\ref{t-u-Eq}) twice gives
\begin{eqnarray}
&&G^{ij}U^{t}_{ijpp}+G_k^{ij,rs}U^{t}_{ijp}U^{t}_{rsp}+\sum_{l=0}^{k-2}\beta_l(x,
u,
t)G_l^{ij,rs}U^{t}_{ijp}U^{t}_{rsp}\notag\\&&+2\sum_{l=0}^{k-2}\big[\beta_l(x,
u,
t)\big]_pG_{l}^{ij}U^{t}_{ijp}+\sum_{l=0}^{k-2}\big[\beta_{l}(x, u, t)\big]_{pp}G_l\notag\\
&=&-[t\alpha \exp{(2u)}]_{pp}.\notag
\end{eqnarray}
Then using the concavity of $G_k$ in $\Gamma_{k-1}$ (see
\cite{HS99}), we deduce that $G_k^{ij,rs}U_{ijp}U_{rsp}\leq0$.
Hence,
\begin{eqnarray}\label{S-1}
G^{ij}U^{t}_{ijpp}\notag&\geq&-\sum_{l=0}^{k-2}\beta_l(x, u,
t)G_l^{ij,rs}U^{t}_{ijp}U^{t}_{rsp}-2\sum_{l=0}^{k-2}\big[\beta_l(x,
u,
t)\big]_pG_{l}^{ij}U^{t}_{ijp}\notag\\&&-\sum_{l=0}^{k-2}\big[\beta_{l}(x,
u, t)\big]_{pp}G_l-[t\alpha \exp{(2u)}]_{pp}.
\end{eqnarray}
Moreover, using the concavity of
$[\frac{\sigma_{k-1}}{\sigma_l}]^{\frac{1}{k-1-l}}$ in
$\Gamma_{k-1}$ for $0\leq l\leq k-2$ (see also (3.10) in \cite{GZ19}
or Chapter XV in \cite{Li96}), we obtain for $0\leq l\leq k-2$
\begin{eqnarray}
-G_{l}^{ij,rs}U_{ijp}U_{rsp} \ge
-\big(1+\frac{1}{k-1-l}\big)G_{l}^{-1}G_{l}^{ij}G_{l}^{rs}U_{ijp}U_{rsp}.\label{93}
\end{eqnarray}
By virtue of (\ref{93}), it yields
\begin{eqnarray*}
&&\sum_{l=1}^{k-2}\beta_lG_{l}^{ij,rs}U^{t}_{ijp}U^{t}_{rs
p}+2\sum_{l=0}^{k-2}[\beta_{l}]_pG_{l}^{ij}U^{t}_{ij p}\\&\leq&
\sum_{l=1}^{k-2}\beta_l\big(1+\frac{1}{k-1-l}\big)G_l^{-1}(G_l^{ij}U^{t}_{ij
p})^2 +2\sum_{l=0}^{k-2}[\beta_{l}]_pG_{l}^{ij}U^{t}_{ij p}
\\&=&\frac{k-l}{k-1-l}
\sum_{l=1}^{k-2}\beta_lG_l^{-1}\bigg(G_l^{ij}U^{t}_{ij
p}+\frac{1}{1+\frac{1}{k-1-l}}\frac{[\beta_l]_p}{\beta_l}G_l\bigg)^2+
\sum_{l=0}^{k-2}\frac{1}{1+\frac{1}{k-1-l}}\frac{([\beta_l]_p)^2}{\beta_l}G_l\\&\leq&
\sum_{l=0}^{k-2}\frac{1}{1+\frac{1}{k-1-l}}\frac{([\beta_l]_p)^2}{\beta_l}G_l.
\end{eqnarray*}
Plugging the inequality above into \eqref{S-1}, we arrive
\begin{eqnarray*}
G^{ij}U_{ijpp}\notag&\geq&-\sum_{l=0}^{k-2}\frac{1}{1+\frac{1}{k-1-l}}
\frac{([\beta_l]_p)^2}{\beta_l}G_l-\sum_{l=0}^{k-2}\big[\beta_{l}(x,
u, t)\big]_{pp}G_l-[t\alpha \exp{(2u)}]_{pp}.
\end{eqnarray*}
So, we complete the proof the last inequality.
\end{proof}

At last, we recall Lemma 4 in \cite{V02} or Lemma 4.2 in
\cite{GV03a} as follows.

\begin{lemma}\label{Gama}
Assume that $s_1<s<s_2$. Then we may choose constants $c_1, c_2$,
and $p$ depending only on $s_1$ and $s_2$ so that $\gamma(s) =
c_1(c_2+s)^p$ satisfies
\begin{eqnarray*}
\gamma^{\prime}(s)>0
\end{eqnarray*}
and
\begin{eqnarray*}
\gamma^{\prime \prime}(s)-\gamma^{\prime}(s)^2>\gamma^{\prime}(s).
\end{eqnarray*}
\end{lemma}

Now, we begin to prove the gradient estimate.

\begin{lemma}\label{estn}
Let $\tau<1$, $\alpha_l(x)$ be a positive smooth function on $M$ for
all $0\leq l\leq k-2$ and $\alpha(x)$ be a smooth function on $M$.
Assume $u$ is a solution of (\ref{eq99}) with $\lambda(U^t)
$$\in\Gamma_{k-1}$. Then there exists a constant $C$, depending on $\tau$, $g_0$,
$||\alpha||_{C^2(M)}$, $||u||_{C^0(M)}$, $\inf_{M} \alpha_l$, and
$||\alpha_l||_{C^2(M)}$ with $0\leq l\leq k-2$ such that
\begin{equation}
\sup_{M}|\nabla u|\le C.\label{C1bdn}
\end{equation}
\end{lemma}

\begin{proof}
Consider the auxiliary function
$$Q=(1+ \frac{|\nabla u|^2}{2})e^{\gamma(u)},$$
where $\gamma(u)=c_1(c_2+u)^p$ is the function in Lemma \ref{Gama}.
Assume that $\max_{M} Q$ is attained at a point $\widetilde{x}$.
After an appropriate choice of the normal frame at $\widetilde{x}$,
we may assume that ${U}^{t}_{ij}(x)$ is diagonal at this point.
Hence $G^{ij}$ is diagonal at $\widetilde{x}$. Differentiating $Q$
at the point $\widetilde{x}$ twice, we obtain
\begin{eqnarray}
Q_i (\widetilde{x})=e^{\gamma(u)}\left((1+ \frac{|\nabla
u|^2}{2})\gamma' u_i +\sum_{l}u_l u_{li}\right)=0,\label{108}
\end{eqnarray}
and
\begin{eqnarray}\label{109}
0\geq Q_{ij} (\widetilde{x})&=&e^{\gamma(u)} \bigg((1+ \frac{|\nabla
u|^2}{2})\Big((\gamma') ^2 u_i u_j+ \gamma' u_{ij} +\gamma'' u_i
u_j\Big) \\ \notag&&+\sum_{l}\Big(2\gamma'u_l u_{lj} u_i+u_{lj}
u_{li} + u_l u_{lij}\Big)\bigg).
\end{eqnarray}
Since $G^{ij}$ is positive definite by Proposition
\ref{ellipticconcave} and $Q_{ij}$ is negative definite at
$\widetilde{x}$, we find at $\widetilde{x}$ from (\ref{109}) and
Ricci identity
\begin{eqnarray}
\notag\allowdisplaybreaks\\
0&\ge&\sum_{i}\big( {G}^{ii}+\frac{1-\tau}{n-2}\sum_{p}
G^{pp}g_0^{ii}\big)Q_{ii}(\widetilde{x})\notag\allowdisplaybreaks\\
&\geq&\big( {G}^{ii}+\frac{1-\tau}{n-2}\sum_{p}
G^{pp}g_0^{ii}\big)\notag\\
&&\cdot \left(\sum_{l}u_l u_{lii}+ (1+ \frac{|\nabla
u|^2}{2})\Big([(\gamma')^2+ \gamma'']u_i^2 + \gamma' u_{ii}  \Big)
+\sum_{l}\Big(2\gamma' u_l u_{li} u_i + u_{li} u_{li}\Big) \right)
\notag\allowdisplaybreaks\\
&\ge&\sum_{i}  \big( {G}^{ii} +\frac{1-\tau}{n-2}\sum_p G^{pp}g_0^{ii}\big)\notag\\
&&\cdot \left(\sum_{l}u_l u_{iil}+ (1+ \frac{|\nabla
u|^2}{2})\Big([(\gamma')^2+ \gamma'']u_i^2 + \gamma' u_{ii}  \Big)
+\sum_{l}\Big(2\gamma' u_l u_{li} u_i
+ u_{li} u_{li}\Big)\right)\notag\allowdisplaybreaks\\
&& -C\sum_i  {G}^{ii}|\nabla u|^2. \allowdisplaybreaks\label{1097}
\end{eqnarray}
Moreover, recalling the definition of $U^t$, using (\ref{100}) and
(\ref{99}), we obtain at $\widetilde{x}$
 \begin{eqnarray}
0&\ge&\sum_{i}\sum_{l}u_l
{G}^{ii}\Big(U^{t}_{iil}-\Big[\frac{2-\tau}{2}|\nabla u|^2- u_i^2
\Big]_{l}\Big) \notag\\&&+
\gamma'\sum_{i} {G}^{ii} \left(U^{t}_{ii}-\Big[\frac{2-\tau}{2}|\nabla u|^2-u_{i}^2\Big]\right)\bigg(1+ \frac{|\nabla u|^2}{2}\bigg)\notag\\
&&+\sum_{i}{G}^{ii}\bigg(1+\frac{|\nabla u|^2}{2}\bigg)
\bigg((\gamma')^2+\gamma''\bigg)\bigg(u_i^2+\frac{(1-\tau)}{n-2}|\nabla u|^2\bigg)\notag\\
&&+2\sum_{i}G^{ii}\bigg(\gamma'\sum_{l}u_iu_{li}u_l+\frac{(1-\tau)}{n-2}\gamma'\sum_{p,
l}u_pu_{lp}u_l\bigg)-C\sum_i {G}^{ii}(1+|\nabla
u|^2)\notag\allowdisplaybreaks
\\&\ge&\sum_{i}{G}^{ii}\Big(-(2-\tau)\sum_{p, l}
u_lu_p u_{pl}+2\sum_{l}u_l  u_i u_{il}\notag\\
&&+\gamma'\Big[-\frac{2-\tau}{2}|\nabla u|^2 + u_i^2\Big]\Big[1+ \frac{|\nabla u|^2}{2}\Big]\Big) \notag\\
&&+[(\gamma')^2+\gamma'']\bigg(1+\frac{|\nabla u|^2}{2}\bigg)
\sum_{i}{G}^{ii}\bigg(u_i^2+\frac{1-\tau}{n-2}|\nabla u|^2\bigg)\notag\\
&&+2\sum_i
G^{ii}\bigg(\gamma'\sum_{l}u_iu_{li}u_l+\frac{1-\tau}{n-2}\gamma'\sum_{p,
l}u_pu_{lp}u_l\bigg)\notag\\
&&-C(|\nabla u|^2+1)\sum_i  {G}^{ii}+C(|\nabla
u|^2+1)\Big(\sum_{l=0}^{k-2}G_l-1\Big).
\allowdisplaybreaks\label{1097-1}
\end{eqnarray}
From (\ref{108}), we know at $\widetilde{x}$
\begin{eqnarray*}
\sum_{l}u_l u_{li}=-\gamma'\Big(1+\frac{|\nabla u|^2}{2}\Big)u_i,
\end{eqnarray*}
which implies at $\widetilde{x}$
\begin{eqnarray*}
&&\sum_i{G}^{ii}\Big(-(2-\tau)\sum_{p, l}u_l u_p u_{pl}+2\sum_{l}u_l
u_i u_{il}\Big)\\&=&\gamma' (1+ \frac{|\nabla
u|^2}{2})\sum_i{G}^{ii}\Big((2-\tau)|\nabla u|^2-2u^2_i\Big),
\end{eqnarray*}
and
\begin{eqnarray*}
&&\sum_i
G^{ii}\bigg(\gamma'\sum_{l}u_iu_{li}u_l+\frac{1-\tau}{n-2}\gamma'\sum_{p,
l}u_pu_{lp}u_l\bigg)\\&=&-(\gamma')^2\Big(1+\frac{|\nabla
u|^2}{2}\Big)\sum_i{G}^{ii}\bigg(u_i^2+\frac{1-\tau}{n-2}|\nabla
u|^2\bigg).
\end{eqnarray*}
Then, plugging the two inequalities above into \eqref{1097-1}, we
arrive at $\widetilde{x}$ from Lemma \ref{Gama}
\begin{eqnarray}
0&\ge& \sum_i {G}^{ii}\Big(\gamma'\Big[1+ \frac{|\nabla
u|^2}{2}\Big]\Big[\frac{2-\tau}{2}|\nabla u|^2-u^2_i\Big]\notag
\\&&+\Big[1+\frac{|\nabla
u|^2}{2}\Big](-\gamma'^2+\gamma'')\Big[u_i^2+\frac{1-\tau}{n-2}|\nabla
u|^2\Big]\Big)\notag
\\&&-C(|\nabla u|^2+1)\sum_i  {G}^{ii}+C(|\nabla
u|^2+1)\big(\sum_{l=0}^{k-2}G_l-1\Big)\notag\allowdisplaybreaks\notag\\
&\ge&\gamma'\Big(\frac{1-\tau}{n-2}+\frac{2-\tau}{2}\Big)|\nabla
u|^2\bigg(1+ \frac{|\nabla u|^2}{2}\bigg)\sum_{i}{G}^{ii}\notag
\\&&-C(|\nabla u|^2+1)\sum_i  {G}^{ii}+C(|\nabla
u|^2+1)\big(\sum_{l=0}^{k-2}G_l-1\Big)\allowdisplaybreaks.\label{105}
\end{eqnarray}
Since $\tau<1$, we have $\frac{1-\tau}{n-2}+\frac{2-\tau}{2}>0$.
Thus, in view of \eqref{G_l-B} and $\gamma^{\prime}>0$, we know the
first term in the right of the inequality \eqref{105} dominates.
Then, absorbing lower order terms results in
\begin{eqnarray*}
C\geq|\nabla u|^2.
\end{eqnarray*}
So, the gradient estimate is immediate.
\end{proof}

\subsection{$C^2$ estimate}

\begin{lemma}\label{estn2}
Let $\tau<1$, $\alpha_l(x)$ be a positive smooth function on $M$ for
all $0\leq l\leq k-2$ and $\alpha(x)$ be a smooth function on $M$.
Assume $u$ is a solution of (\ref{eq99}) with $\lambda(U^t)
$$\in\Gamma_{k-1}$. Then there exists a constant $C$, depending on $\tau$,
$g_0$,
$||\alpha||_{C^2(M)}$, $\||u||_{C^1(M)}$, $\inf_{M}\alpha_l$, and
$||\alpha_l||_{C^2(M)}$ with $0\leq l\leq k-2$ such that
\begin{equation}
\sup_{M} |\nabla^2 u|\le C.\label{C2bdn}
\end{equation}
\end{lemma}

\begin{proof}
Since $\lambda(U^t)\in\Gamma_2$, we have
\[|U^t_{ij}|\le C tr U^t.\]
Therefore,
\begin{equation}
|u_{ij}|\le C(\Delta u+1).\label{96}
\end{equation}
So we only estimate $\Delta u$. Thus, we take the auxiliary function
$$H(x)=(\Delta u+\mu|\nabla u|^2),$$
where $\mu$ is a positive constant which will be chosen later.
Assume $x_0$ is the maximum point of $H$. After an appropriate
choice of the normal frame at $x_0$, we further assume $U^{t}_{ij}$
and hence $G^{ij}$ is diagonal at the point $x_0$. Then we have at
$x_0,$
\begin{equation}\label{97}
H_i(x_0)=\sum_{k}(u_{kki}+2\mu u_ku_{ki})=0,
\end{equation}
and
\begin{equation}\label{9701}
H_{ii}(x_0)=\sum_{k}\big(u_{kkii}+ 2\mu u_k u_{kii} +2\mu
u_{ki}^2\big)\leq0.
\end{equation}
From the positivity of $G^{ij}$ and (\ref{9701}), we arrive at $x_0$
\begin{eqnarray}
0&\ge&\sum_i\bigg( {G}^{ii} +\frac{1-\tau}{n-2}\sum_{p}
G^{pp}g_0^{ii}\bigg)H_{ii}(x)\notag\\&\ge&\sum_i\bigg( {G}^{ii}
+\frac{1-\tau}{n-2}\sum_{p}
G^{pp}g_0^{ii}\bigg)\sum_{k}\big(u_{kkii}+ 2\mu u_k u_{kii} +2\mu
u_{ki}^2\big)\notag
\\&\ge&\sum_i\bigg( {G}^{ii} +\frac{1-\tau}{n-2}\sum_{p}
G^{pp}g_0^{ii}\bigg)\sum_{k}\big(u_{iikk}+2\mu u_ku_{iik}+2\mu
u_{ki}^2-C\Delta u\big),\notag\allowdisplaybreaks
\end{eqnarray}
where we use Ricci identity to get the last inequality. In view of
(\ref{96}), we may assume $\Delta u$ is large enough. Thus it
follows from the definition of $U^t$ and (\ref{97}) that
\begin{eqnarray}
0&\ge&\sum_iG^{ii}\sum_{p}\bigg(U^{t}_{iipp}+(u_i^2)_{pp}-\Big[\frac{2-\tau}{2}|\nabla
u|^2
\Big]_{pp}\notag\allowdisplaybreaks\\
&&+2\mu u_p\Big(U_{iip}+(u_i^2)_p-\Big[\frac{2-\tau}{2}|\nabla
u|^2\Big]_p\Big)+ 2\mu
u_{pi}^2+\frac{2\mu(1-\tau)}{n-2}\sum_{l}u^2_{lp}\bigg)
\notag\allowdisplaybreaks\\
&&-C\sum_{i}G^{ii}\Delta u\notag\allowdisplaybreaks\\
&\ge&\sum_iG^{ii}\sum_{p}\bigg(U^{t}_{iipp}+2(-2\mu
u_iu_{ip}u_{p}+u^2_{ip})
-(2-\tau)\sum_{l}\big({-2\mu u_lu_{lp}u_p+u_{lp}^2}\big)\notag\allowdisplaybreaks\\
&&+2\mu
u_p\Big[(U^t_{iip}+2u_iu_{ip}-(2-\tau)\sum_{l}u_lu_{lp}\Big]+2\mu
u_{pi}^2 +\frac{2\mu(1-\tau)}{n-2}\sum_{l}u^2_{lp}\bigg)
\notag\allowdisplaybreaks\\&&-C\sum_iG^{ii}(1+\Delta u)\notag\allowdisplaybreaks\\
&\ge&\bigg(\Big[\frac{2\mu(1-\tau)}{n-2}-2+\tau\Big](\Delta
u)^2-C\Delta u-C\bigg)\sum_{i}G^{ii}
\notag\allowdisplaybreaks\\&&+\sum_{i,p}G^{ii}U^t_{iipp}+2\mu
\sum_{i,p}u_pG^{ii}U^t_{iip}.\notag
\end{eqnarray}
Then using (\ref{99}) and (\ref{94}), we deduce that
\begin{eqnarray}\label{Chen}
0&\ge&\bigg(\Big[\frac{2\mu(1-\tau)}{n-2}-(2-\tau)\Big](\Delta
u)^2-C\Delta
u-C\bigg)\sum_{i}G^{ii}\\&&+C\Big(\sum_{l=0}^{k-2}G_l+1\Big)\big(\Delta
u+1\big).
\end{eqnarray}
Since $\tau<1$, we may choose $\mu$ large to dominate the $(2-\tau)$
term (this is the point where the assumption $\tau<1$ is crucial).
Choosing $\mu>\frac{(2-\tau)(n-2)}{2(1-\tau)}$ and using the
inequality \eqref{G_l-B}, we conclude at $x_0$ from the inequality
above
\begin{eqnarray*}
C\geq |\Delta u|^2.
\end{eqnarray*}
So, we complete the proof.
\end{proof}

\subsection{Proof of Theorem \ref{main2}} \

In this section, we use the degree theory for nonlinear elliptic
equation developed in \cite{Li89} to prove Theorem \ref{main2}.
After establishing the  a priori estimates Lemma \ref{estn0}, Lemma
\ref{estn}, Lemma \ref{estn2}, we know that the equation
\eqref{t-u-Eq} is uniformly elliptic if we notice \eqref{G_l-B} for
the case $l=0$
\begin{eqnarray}\label{l=0}
\sigma_{k-1}(U^t)\ge C>0.
\end{eqnarray}
From Evans-Krylov estimates \cite{Eva82, Kry83}, and Schauder
estimates, we have
\begin{eqnarray}\label{C2++}
|u|_{C^{4,\delta}(M)}\leq C
\end{eqnarray}
for any solution $u$ to the equation \eqref{t-u-Eq} with
$\lambda(U^t)\in \Gamma_{k-1} $, where $0<\delta<1$.  Recalling the
equation \eqref{t-u-Eq}
\begin{eqnarray*}
G(U^t)&:=&\frac{\sigma_k\big(U^t\big)}
{\sigma_{k-1}\big(U^t\big)}\notag\allowdisplaybreaks-\frac{\sum_{l=0}^{k-2}\Big([(1-t)c+t\alpha_l(x)]\exp[{2(k-l)u}]\sigma_{l}\big(U^t\big)\Big)}
{\sigma_{k-1}\big(U^t\big)}\notag\allowdisplaybreaks\\
&=&-t\alpha \exp({2u}),
\end{eqnarray*}
where $ t\in [0, 1], e=(1,\cdots, 1)$,
\begin{eqnarray}\label{c}
c=\frac{\sigma_{k}(e)}{\sum_{l=0}^{k-2}\sigma_l(e)},
\end{eqnarray}
and
\[U^t=\nabla^2 u +\frac{1-\tau}{n-2} \triangle u g_0 +
\frac{2-\tau}{2}|\nabla u|^2 g_0- du\otimes du
-tA^{\tau}_{g_0}+(1-t)g_0.\] Then we consider a family of the
mappings for $t \in [0, 1]$
$$F(.; t):
C_{0}^{4, \delta}(M)\rightarrow C^{2, \delta}(M),$$ which is defined
by
\begin{eqnarray*}
F(u; t):=G(U^t)+t\alpha\exp(2u),
\end{eqnarray*}
where
\begin{eqnarray*}
C_{0}^{4, \delta}(M)=\{u \in C^{4,\delta}(M): \lambda(U^t)
\in\Gamma_{k-1} \}
\end{eqnarray*}
is an open subset of $C^{4,\delta}(M)$. Let $$\mathcal{O}_R=\{u \in
C_{0}^{4, \delta}(M): |u|_{C^{4,\delta}(M)}<R\},$$ which clearly is
also an open subset of $C^{4,\delta}(M)$. Moreover, if $R$ is
sufficiently large, $F(u; t)=0$ has no solution on $\partial
\mathcal{O}_R$ by \eqref{l=0} and the a prior estimate established
in \eqref{C2++}. Therefore the degree $\deg(F(.; t), \mathcal{O}_R,
0)$ is well-defined for $0\leq t\leq 1$. Using the homotopic
invariance of the degree, we have
\begin{eqnarray*}
\deg(F(.; 1), \mathcal{O}_R, 0)=\deg(F(.; 0), \mathcal{O}_R, 0).
\end{eqnarray*}
When $t=0$, (\ref{t-u-Eq}) becomes
\begin{eqnarray}
\sigma_k\big(U^0\big) -c
\sum_{l=0}^{k-2}\exp[{2(k-l)u}]\sigma_{l}\big(U^0\big)=0\label{107}
\end{eqnarray}
with
\[U^0=\nabla^2 u +\frac{1-\tau}{n-2} \triangle u g_0 +
\frac{2-\tau}{2}|\nabla u|^2 g_0- du\otimes du+g_0.\]

\begin{lemma}\label{t=0}
$u=0$ is the unique solution for (\ref{107}).
\end{lemma}

\begin{proof}
Assume $x$ and $y$ are the maximum and minimum points of $u$
respectively. Then we obtain by (\ref{107}),
\[\sigma_k(e)\leq c
\sum_{l=0}^{k-2}\exp[{2(k-l)u(x)}]\sigma_{l}(e),\] which implies by
the definition \eqref{l=0} of $c$
\begin{eqnarray*}
u(x)\geq 1.
\end{eqnarray*}
Similarly, we have
\[\sigma_k(e)\geq c
\sum_{l=0}^{k-2}\exp[{2(k-l)u(x)}]\sigma_{l}(e),\] which implies
\begin{eqnarray*}
u(y)\leq 1.
\end{eqnarray*}
Thus $u\equiv 0$.
\end{proof}

Lemma \ref{t=0} shows that $u=0$ is the unique solution to the
equation \eqref{t-u-Eq} for $t=0$. Let $u(x, s)$ be the variation of
$u=0$ such that $u'_{s}=\phi$ at $s=0$. Then
\begin{eqnarray*}
\delta_{\varphi}F(0; 0)=a_{ij}\phi_{ij}+\textrm{1st\ derivatives\ in
\
}\phi-c\frac{\sum_{l=0}^{k-2}2(k-l)\sigma_l(e)}{\sigma_{k-1}(e)}\phi,
\end{eqnarray*}
where $a_{ij}$ is a positive definite matrix and $\delta F(0; 0)$ is
the linearized operator of $F$ at $u=0$. Clearly, $\delta F(0; 0)$
is an invertible operator. Therefore,
\begin{eqnarray*}
\deg(F(.; 1), \mathcal{O}_R; 0)=\deg(F(.; 0), \mathcal{O}_R, 0)=\pm
1.
\end{eqnarray*}
So, we obtain a solution at $t=1$. This completes the proof of
Theorem \ref{main2}.


\bigskip

\begin{thebibliography}{0}


\bibitem{Au76} T. Aubin,
Equations differentielles non lineaires et problme de Yamabe
concernant la courbure scalaire, J. Math. Pures Appl. 55(9) (1976)
269-296.

\bibitem{BG08} T.P. Branson, A.R. Gover, Variational status of a class of fully
nonlinear curvature prescription problems. Calc. Var. Partial
Differ. Equ. 32(2) (2008) 253-262.

\bibitem{BV04} S. Brendle, J. Viaclovsky, A variational characterization
for $\sigma_{\frac{n}{2}}$, Calc. Var. Partial Differ. Equ. 20
(2004) 399-402.

\bibitem{BM09} S. Brendle, F.C. Marques, Blow-up phenomena for the Yamabe equations II, $25\le n\le 51$.
J. Diff. Geom. 81 (2009) 225-250.

\bibitem{CNS84} L.A. Caffarelli, L. Nirenberg, J. Spruck, The Dirichlet problem for
nonlinear second-order elliptic equations, I: Monge-Amp\'ere
equation, Comm. Pure and Appl. Math. 37 (1984) 369-402.

\bibitem{CNS85} L.A. Caffarelli, L. Nirenberg, J. Spruck,
Dirichlet problem for nonlinear second order elliptic equations III,
Functions of the eigenvalues of the Hessian, Acta Math. 155 (1985)
261-301.

\bibitem{CGY02b} S.Y.A Chang, M.J. Gursky, P.C. Yang, An equation of Monge-Am\'ere type
in conformal geometry, and four-manifolds of positive Ricci
curvature, Ann. of Math. 155(2) (2002) 709-787.

\bibitem{CGY02a} S.Y.A Chang, M.J. Gursky, P.C. Yang,
An a priori estimate for a fully nonlinear equation on
four-manifolds, J. Anal. Math. 87 (2002) 151-186.

\bibitem{CGH} L. Chen, X. Guo, Y. He,
A class of fully nonlinear equations arising in conformal geometry,
preprint.

\bibitem{Chen05} S. Chen, Local estimates for some fully nonlinear elliptic
equations, Int. Math. Res. Not. 55 (2005) 3403-3425.

\bibitem{Chen07} S. Chen, Boundary value problems for some fully nonlinear elliptic
equations, Calc. Var. Partial Differ. Equ. 30(1) (2007) 1-15.

\bibitem{Chen09} S. Chen, Conformal deformation on manifolds with
boundary, Geom. Funct. Anal. 19(4),(2009) 1029-1064.

\bibitem{E92} J.F. Escobar, The Yamabe problem on manifolds with
boundary, J. Diff. Geom. 35(1) (1992) 21-84.

\bibitem{Eva82} L.C. Evans, Classical solutions of fully nonlinear, convex, second-order
elliptic equations, Comm. Pure Appl. Math. 35(3) (1982) 333-363.

\bibitem{Fu07} J.X. Fu, S.T. Yau,
A Monge-Amp\'ere type equation motivated by string theory, Comm.
Anal. Geom. 15 (2007) 29-76.

\bibitem{Fu08} J.X. Fu, S.T. Yau, The theory of superstring with flux on
non-K$\ddot{a}$hler manifolds and the complex Monge-Ampere equation,
J. Diff. Geom. 78 (2008) 369-428.

\bibitem{Ger06} C. Gerhardt, Curvature problems, Series in Geometry and Topology, vol. 39,
International Press of Boston Inc., Sommerville, 2006.

\bibitem{G07} B. Guan, Conformal metrics with prescribed curvature curvature functions on manifolds with
boundary, Amer. J. Math. 129(4) (2007) 915-942.

\bibitem{G08} B. Guan, Complete conformal metrics of negative Ricci curvature on compact
manifolds with boundary, Int. Math. Res. Not. IMRN. (2008)
https://doi.org/10.1093/imrn/rnn105.

\bibitem{GLW10} Y.X. Ge, C.S. Lin, G.F. Wang, On the $\sigma_2$-scalar
curvature, J. Diff. Geom. 84(1) (2010) 45-86.

\bibitem{GW13} Y.X. Ge, G.F. Wang, On a conformal quotient equation. II, Comm.
Anal. Geom. 21(1) (2013) 1-38.

\bibitem{GW07} Y.X. Ge, G.F. Wang, On a conformal quotient equation, Int.
Math. Res. Not. 2007 (6), Art. ID rnm019, 32 pp.

\bibitem{GW06} Y.X. Ge, G.F. Wang, On a fully nonlinear Yamabe problem, Ann. Sci.
Ecole Norm. Sup. 39(4) (2006) 569-598.

\bibitem{GLN18} M. Gonz\'alez, Y.Y. Li, L. Nguyen,
Existence and uniqueness to a fully nonlinear version of the
Loewner-Nirenberg problem, Commun. Math. Stat. 6(3) (2018) 269--288.

\bibitem{GW03a} P.F. Guan, G.F. Wang,
Local estimates for a class of fully nonlinear equations arising
from conformal geometry, Int. Math. Res. Not. 26 (2003) 1413-1432.

\bibitem{GW03b} P.F. Guan, G.F. Wang.
A fully nonlinear conformal flow on locally conformally flat
manifolds, J. Reine Angew. Math. 557 (2003) 219-238.

\bibitem{GZ19} P.F. Guan, X.W. Zhang, A class of curvature type equations, to appear in Pure and
Applied Math Quarterly, preprint.arXiv:1909.03645.

\bibitem{GSW11} M.J. Gursky, J. Streets, M. Warren,
Existence of complete conformal metrics of negative Ricci curvature
on manifolds with boundary, Calc. Var. Partial Differ. Equ. 41(1-2)
(2011) 21-43.

\bibitem{GV03a} M.J. Gursky, J. Viaclovsky, Fully nonlinear equations on Riemannian manifolds
with negative curvature, Indiana U. Math. J. 52 (2003) 399-419.

\bibitem{GV07} M.J. Gursky, J. Viaclovsky,
Prescribing symmetric functions of the eigenvalues of the Ricci
tensor, Ann. Math. 166 (2007) 475-531.

\bibitem{HL} R. Harvey,  H. Lawson, Calibrated geometries,
Acta. Math. 148 (1982) 47-157.

\bibitem{HSh11} Y. He,  W.M. Sheng,
On existence of the prescribing $k$-curvature problem on manifolds
with boundary, Comm. Anal. Geom. 19 (2011) 53-77.

 \bibitem{HSh13} Y. He, W.M. Sheng,
 Local estimates for some elliptic equations arising from conformal
 geometry, Int. Math. Res. Not. 2 (2013) 258-290.

\bibitem{HS99} G. Huisken, C. Sinestrari,
Convexity estimates for mean curvature flow and singularities of
mean convex surfaces, Acta Math. 183(1) (1999) 45-70.

\bibitem{JT19} F.D. Jiang, N.S. Trudinger,
Oblique boundary value problems for augmented Hessian equations III,
Comm. Partial Differential Equations 44(8) (2019) 708-748.

\bibitem{JT18} F.D. Jiang, N.S. Trudinger,
Oblique boundary value problems for augmented Hessian equations I,
Bull. Math. Sci. 8(2) (2018) 353-411.

\bibitem{JT17} F.D. Jiang, N.S. Trudinger,
Oblique boundary value problems for augmented Hessian equations II,
Nonlinear Anal. 154 (2017) 148-173.

\bibitem{JLL07} Q.N. Jin, A. Li, Y.Y. Li, Estimates and existence results for a fully nonlinear Yamabe problem
on manifolds with boundary, Calc. Var. Partial Differ. Equ. 28
(2007) 509-543.

\bibitem{Kry83} N.V. Krylov,
Boundedly inhomogeneous elliptic and parabolic equations in a
domain, Izv. Akad. Nauk SSSR Ser. Mat. 47(1) (1983) 75-108.

\bibitem{Kry95} N.V. Krylov,  On the general notion of fully nonlinear second order elliptic
equation, Trans. Amer. Math. Soc. 347(3) (1995) 857-895.

\bibitem{KMS09} M.A. Khuri, F.C. Marques, R.M. Schoen, A compactness theorem for the Yamabe
problem, J. Diff. Geom. 81 (2009) 143-96.

\bibitem{LL03} A.B. Li, Y.Y. Li,
On some conformally invariant fully nonlinear equations. Comm. Pure
Appl. Math. 56 (2003) 1416-1464.

\bibitem{LL06} A.B. Li, Y.Y. Li,
A fully nonlinear version of the Yamabe problem on manifolds with
boundary, J. Eur. Math. Soc. 8 (2006) 295-316.

\bibitem{LiSh05} J.Y. Li, W.M. Sheng,
Deforming metrics with negative curvature by a fully nonlinear flow,
Calc. Var. Partial Differ. Equ. 23 (2005) 33-50.

\bibitem{Li89} Y.Y. Li, Degree theory for second order nonlinear elliptic operators and its
applications, Comm. Partial Differential Equations 14 (1989)
1541-1578.

\bibitem{LN14} Y.Y. Li, L. Nguyen,
A compactness theorem for a fully nonlinear Yamabe problem under a
lower Ricci curvature bound, J. Funct. Anal. 266(6) (2014)
3741-3771.

\bibitem{Li96} G. Lieberman, Second order parabolic differential equations. World Scientific, 1996.

\bibitem{LT94} M. Lin M, N.S. Trudinger,
On some inequalities for elementary symmetric functions. Bull. Aust.
Math. Soc. 50 (1994) 317-326.

\bibitem{Ph17} D. Phong, S. Picard, X.W. Zhang, The Fu-Yau equation with negative
slope parameter, Invent. Math. 209 (2017) 541-576.

\bibitem{Ph19} D. Phong, S. Picard, X.W. Zhang, On estimates for the Fu-Yau
generalization of a Strominger system, J. Reine Angew. Math. 751
(2019) 243-274.

\bibitem{Ph20} D. Phong, S. Picard, X.W. Zhang, Fu-Yau Hessian
Equations, arXiv:1801.09842, to appear in J. Diff. Geom.

\bibitem{S84} Schoen R,
Conformal deformation of a Riemannian metric to constant scalar
curvature, J. Diff. Geom. 20 (1984) 479-495.

\bibitem{Sch13} R. Schneider, Convex bodies: the Brunn-Minkowski theory, Second edition, No. 151.
Cambridge University Press, 2013.

\bibitem{STW07} W.M. Sheng , N.S. Trudinger, X.J. Wang.
 The Yamabe problem for higher order curvatures,
J. Differ. Geom. 77, (2007) 515-553.

\bibitem{ShY13} W.M. Sheng, L.X. Yuan,
The $k$-Yamabe flow on manifolds with boundary, Nonlinear Anal. 82
(2013) 127-141.

\bibitem{ShZ07} W.M. Sheng, Y. Zhang, A class of fully nonlinear equations
arising from conformal geometry, Math. Z. 255(1) (2007) 17-34.

\bibitem{Sui17} Z.N. Sui, Complete conformal metrics of negative Ricci curvature on Euclidean
spaces, J. Geom. Anal. 27(1) (2017) 893-907.

\bibitem{Tr90} N.S. Trudinger,
The Dirichlet problem for the prescribed curvature equations, Arch.
Rational Mech. Anal. 111(2) (1990) 153-179.

\bibitem{Tr68} N.S. Trudinger,
Remarks concerning the conformal deformation of Riemannian
structures on compact manifolds, Ann. Scuola Norm. Sup. Pisa 22(3),
(1968) 265-274.

\bibitem{V00} J. Viaclovsky, Conformal geometry, contact geometry, and the calculus of
variations, Duke Math. J. 101 (2000) 283-316.

\bibitem{V00-1} J. Viaclovsky,
Conformally invariant Monge-Amp\'ere equations: global solutions,
Trans. Amer. Math. Soc. 352(9) (2000) 4371-4379.

\bibitem{V02} J. Viaclovsky,
Estimates and existence results for some fully nonlinear elliptic
equations on Riemanian manifolds, Comm. Anal. Geom. 10(4) (2002)
815-846.

\bibitem{W06} X.J. Wang,
A priori estimates and existence for a class of fully nonlinear
elliptic equations in conformal geometry, Chinese Ann. Math. B. 27,
(2006) 1-10.

\bibitem{Y60} H. Yamabe,
On a deformation of Riemannian structures on compact manifolds,
Osaka Math. J. 12 (1960) 21-37.

\end{thebibliography}
\end{document}